\documentclass[11pt,reqno]{amsart}

\usepackage{a4wide}
\usepackage{amsmath,amsfonts,amssymb,
	amsthm, color, MnSymbol,relsize}
\usepackage{hyperref}
\usepackage{csquotes}
\usepackage{soul}
\usepackage{cancel}

\usepackage[active]{srcltx}

\usepackage{mathrsfs}
\usepackage{multicol}

\newtheorem{theorem}{Theorem}[section]

\newtheorem{proposition}[theorem]{Proposition}

\newtheorem{lemma}[theorem]{Lemma}

\theoremstyle{definition}
\newtheorem{remark}[theorem]{Remark}

\newtheorem{notation}[theorem]{Notation}

\def\R{\mathbb{R}}

\def\dH{\text{\tt {\bf d}}_\text{\tt {\bf H}}}
\def\Op{\mathfrak{Op}}

\def\z{\mathfrak{z}}
\def\Rd{\mathbb{R}^d}
\def\Z{\mathbb{Z}}
\def\bb1{{\rm{1}\hspace{-3pt}\mathbf{l}}}
\def\supp{\text{\sf supp}}

\def\dist{\text{\sf dist}}

\def\mW{\mathfrak{W}}

\def\mz{\mathring{z}}

\def\beq{\begin{equation}}
	\def\eeq{\end{equation}}

\numberwithin{equation}{section}

\title{Spectral regularity with respect to dilations for a class of pseudodifferential operators.}

\author[H. D. Cornean, R, Purice]{Horia D. Cornean \and R. Purice}

\begin{document}

{\begin{abstract} {
We continue the study of the perturbation problem discussed in \cite{CP3} and get rid of the 'slow variation' assumption by considering symbols of the form $a\big(x+\delta\,F(x),\xi\big)$ with $a$ a real H\"{o}rmander symbol of class $S^0_{0,0}(\Rd\times\Rd)$ and $F$ a smooth function with all its derivatives globally bounded, with $|\delta|\leq1$. We prove that while the Hausdorff distance between the spectra of the Weyl quantization of the above symbols in a neighbourhood of $\delta=0$ is still of the order $\sqrt{|\delta|}$, the distance between their spectral edges behaves like $|\delta|^\nu$ with $\nu\in[1/2,1)$ depending on the rate of decay of the second derivatives of $F$ at infinity.}
\end{abstract}
}
\maketitle

\section{Introduction}

In \cite{CP3} we {considered the regularity of} {\  investigated how }the spectrum of a class of $\Psi$DO {\  changes (seen as a subset of the real line)} with respect to a family of slowly varying dilation-type perturbations. {\  More precisely, we worked with symbols of the form $a\big(x+F(\delta \, x),\xi\big)$ with $a$ a real H\"{o}rmander symbol of class $S^0_{0,0}(\Rd\times\Rd)$ and $F$ a smooth function with all its derivatives globally bounded, and with $0<|\delta|\leq 1$.  The motivation came from a related problem discussed in \cite{GRS} in which $F$ was an affine function.} In this note we present some results that may be obtained when one eliminates the 'slow variation' hypothesis {\  and the perturbed symbols are of the form $a\big(x+\delta\,F(x),\xi\big)$ instead of $a\big(x+F(\delta \, x),\xi\big)$. We note that when $F$ is affine as in \cite{GRS}, the two problems are essentially the same}.

We shall use the multi-index conventions of \cite{H-3}. Let:
\begin{align}\label{dc1}
	\nu_{n,m}(a)\,:=\,\underset{|\alpha|\leq n}{\max}\,\underset{|\beta|\leq m}{\max}\,\sup_{(x,\xi) \in \R^{2d}}|\partial_x^\alpha\partial_\xi^\beta a|,\quad \forall (n,m)\in \mathbb{N}\times\mathbb{N}
\end{align}
and $S^0_{0,0}(\R^{d}\times\R^d)$ \cite{H-3} the set of smooth functions  satisfying
\begin{align}\label{dc1-1}
	\nu_{n,m}(a)\,<\infty,\quad \forall (n,m)\in \mathbb{N}\times\mathbb{N}.
\end{align}
We shall consider some real H\"ormander symbol $a(x,\xi)$ of class $S^0_{0,0}(\R^{d}\times\R^d)$.

We shall denote by $\big(\cdot,\cdot\big)_{L^2(\R^d)}$ the scalar product in $L^2(\R^d)$ (considered antilinear in the first variable), with the quadratic norm denoted simply by $\|\cdot\|$ and we shall use the notation $\big\langle\cdot,\cdot\big\rangle_{\mathscr{S}(\mathcal{V})}:\mathscr{S}^\prime(\mathcal{V})\times\mathscr{S}(\mathcal{V})\rightarrow\mathbb{C}$ for the canonical bilinear duality map for tempered distributions on the real finite dimensional Euclidean space $\mathcal{V}$.

Following \cite{H-3} we define the Weyl quantization of the symbol $a\in\,S^0_{0,0}(\R^{d}\times\R^d)$ as the operator:
\beq\label{DF-Op-w-a}
\big(\Op^w(a)\varphi\big)(x)=(2\pi)^{-d}\int_{\R^d}dy\int_{\R^d}d\eta\,e^{i<\eta,x-y>}\,a\big((x+y)/2,\eta\big)\,\varphi(y),\quad\forall\varphi\in\mathscr{S}(\R^d),\ \forall x\in\R^d.
\eeq
Due to the Calderon-Vaillancourt Theorem (see \cite{CV} and \S XIII.1 in \cite{T}) this operator is bounded in $L^2(\R^d)$ with the following bound on the operator norm:
\beq\label{F-CV-est}
\big\|\Op^w(a)\varphi\big\|_{\mathbb{B}(L^2(\R^d))}\,\leq\,C\,\nu_{3d+4,3d+4}(a).
\eeq 
We shall use the same notation for its extension to the entire Hilbert space. Let $\mathfrak{K}[a]\in\mathscr{S}^\prime(\R^d\times\R^d)$ be the distribution kernel of $\Op^w(a)$ (see \cite{Schw}); it may be computed by the following formula: 
\beq
\mathfrak{K}[a]:=(2\pi)^{-d/2}\big((\bb1\otimes\mathcal{F}^-)a\big)\circ\Upsilon
\eeq
where $\Upsilon:\R^d\times\R^d\ni(x,y)\mapsto\big((x+y)/2,x-y\big)\in\R^d\times\R^d$  is a bijection with Jacobian -1 and:
\beq
\big(\mathcal{F}^-\varphi\big)(v)\,:=\,(2\pi)^{-d/2}\int_{\R^d}d\xi\,e^{i<\xi,v>}\,\varphi(\xi),\quad\forall\varphi\in\mathscr{S}(\R^d),\ \forall v\in\R^d
\eeq
is the inverse Fourier transform.
We also define the distribution $\widetilde{\mathfrak{K}}[a]:=\mathfrak{K}[a]\circ\Upsilon^{-1}\in\mathscr{S}^\prime(\R^d\times\R^d)$.
With a slight abuse, we can write the following explicit formula:
\beq
\mathfrak{K}[a](z+v/2,z-v/2)\equiv\widetilde{\mathfrak{K}}[a](z,v):=(2\pi)^{-d}\int_{\R^d}d\eta\,e^{i<\eta,v>}\,a(z,\eta).
\eeq
\begin{proposition}\label{P-distr-ker} (see \cite{H-3}) 
The tempered distribution $\widetilde{\mathfrak{K}}[a]\in\mathscr{S}^\prime(\R^d\times\R^d)$ is in fact a smooth (with respect to the weak topology) distribution valued function $\Rd\ni\,z\mapsto\widetilde{\mathfrak{K}}[a](z,\cdot)\in\mathscr{S}^\prime(\R^d_v)$ such that for any $z\in\R^d$ the distribution $\widetilde{\mathfrak{K}}[a](z,\cdot)\in\mathscr{S}^\prime(\R^d_v)$ has singular support contained in $\{v=0\}$ (possibly void) and rapid decay together with all its derivatives, in the complement of $v=0$.
\end{proposition}

\paragraph{$\boldsymbol{L^2(\R^d)}$-boundedness criterion.} Given a distribution kernel $\widetilde{\mathfrak{K}}[a]$ as in the above Proposition \ref{P-distr-ker} and using the operator-norm estimation in the Calderon-Vaillancourt Theorem for its associated H\"{o}rmander symbol $a=(\bb1\otimes\mathcal{F})\widetilde{\mathfrak{K}}[a]\in S^0_{0,0}(\Rd\times\Rd)$, our main criterion for $L^2(\R^d)$-boundedness of the associated linear operator in $L^2(\R^d)$ will be the boundedness {\  of at least one of the seminorms}:
\beq\label{Cr-Bd}
\nu_{n,m}\big((\bb1\otimes\mathcal{F})\widetilde{\mathfrak{K}}\big),\quad\min(n,m)\geq3d+4.
\eeq
\begin{notation}
We shall use the following notations:
\begin{itemize}
	\item $<v>:=\sqrt{1+|v|^2}$, for any $v\in\R^d$ and $\mathfrak{s}_p(v):=<v>^p$ for any $p\in\R$.
	\item $\tau_z$ for the translation with $-z\in\R^d$ on any space of functions or distributions on $\R^d$.
	\item $C^\infty_1(\R^d;\Rd)$ defined as the space of smooth $\R^d$-valued functions with bounded derivatives of all strictly positive orders.
\end{itemize}
\end{notation}

\paragraph{The Problem.}
Let $F\in C^\infty_1(\R^d;\Rd)$ and 
$\delta\in\R$ with $|\delta|\leq1$.
To any real-valued symbol $a\in S^0_{0,0}(\R^d\times\R^d)$ we associate the perturbed symbols:\vspace*{-0,3cm}
\beq\begin{split}\nonumber
	&a[F]_\delta(x,\xi)\,:=\,a\big(x+\delta\,F(x),\xi\big).
\end{split}\eeq
We are interested in the variation of the spectrum $\sigma\big(\Op^w(a[F]_\delta)\big)\subset\mathbb{R}$, as a set, when $\delta$ goes to 0.
\begin{remark}\label{R-est-aFdelta}
We evidently have the inequalities:
$$
\nu_{n,m}(a[F]_\delta)\,\leq\,C_n(\delta,F)\nu_{n,m}(a),\quad\forall(n,m)\in\mathbb{N}\times\mathbb{N},
$$
with $C_n(\delta,F)$ depending on the $\sup$-norm of the derivatives of $F$ up to order $n-1$, uniformly in $\delta\in(0,1]$.
\end{remark}
We shall use the short-hand notations (for $|\delta|\leq1$):
\begin{align}\label{F-intkerdelta}
&K_\delta:=\Op^w\big(a[F]_\delta\big)\in\mathbb{B}\big(L^2(\R^d)\big),\quad\mathfrak{K}_\delta:=\widetilde{\mathfrak{K}}\big[a[F]_\delta\big]\in\mathscr{S}^\prime(\R^d\times\R^d);\\\label{F-spedgedelta}
&\mathcal{E}_+(\delta):=\sup\,\sigma(K_\delta).
\end{align}
\textbf{The Hausdorff distance:} $\dH(A,B):=\max\Big\{\underset{t\in A}{\sup}\,\dist(t,B)\,,\,\underset{t\in B}{\sup}\,\dist(t,A)\Big\}$ for $A,B$ subsets of $\mathbb{C}$.

\section{The main results}

\begin{theorem}\label{T-A}
	There exists $C(a,F)>0$ such that for $|\delta|\leq1$ we have the estimation:
	$$
	\dH\Big(\sigma\big(K_\delta\big)\,,\,\sigma\big(K_0\big)\Big)\,\leq\,C(a,F)\sqrt{|\delta|}.
	$$
\end{theorem}
\begin{remark}
	Counter-examples from the literature show that this estimation is 'sharp', i.e. spectral gaps of order $\sqrt{|\delta|}$ may be created by these type of perturbations.
\end{remark}
From Theorem 1.5 in \cite{CP3} and some other similar results from the literature, one may expect a more regular behaviour of the spectral edges {\  seen as functions of $\delta$. By spectral edges we understand quantities like $\sup \sigma(K_\delta)$, $\inf \sigma(K_\delta)$, or the extremities of the possible inner gaps in the spectrum}. In this situation we obtain the following result depending on the decay at infinity of the second order derivatives of the 'perturbing function' $F\in C^\infty_1(\Rd;\Rd)$.
\begin{theorem}\label{T-B}
	Suppose that $\big|\big(\partial_{x_j}\partial_{x_k}F\big)(x)\big|\leq\,C<x>^{-(1+\mu)}$ for some $C>0$, $\mu>0$ and for any pair of indices $(j,k)$. Then there exists $C(a,F)>0$ and $\delta_0>0$ such that for $|\delta|\leq\delta_0$ we have the estimation:
	$$
	\big|\,\mathcal{E}_\pm(\delta)\,-\,\mathcal{E}_\pm(0)\,\big|\,\leq\,C(a,F)|\delta|^{(1+\mu)/(2+\mu)}.
	$$	
\end{theorem}

\section{Proof of Theorem \ref{T-A}}

The main idea of the proof is to construct a 'quasi-resolvent" (see \eqref{DF-qRez-delta}) and use the unitarity of $x$-translations and localization around a lattice of points in $\R^d_x$ in order to control the possible linear growth of $F$.
We notice that the invariance of our arguments when changing $F$ into $-F$ allows us to work with $\delta\geq0$.

Let us consider some exponent $\kappa\in(0,1]$ and the discrete family of points $\Gamma_\delta:=\big\{z_\gamma(\delta):=\delta^\kappa\gamma\in\R^d,\ \gamma\in\Z^d\big\}$. We notice that for any $\mz\in\Gamma_\delta$, the bounded operator $\tau_{-\mz}K_0\tau_{\mz}$ has the integral kernel $\mathfrak{K}_0(z+\mz,v)$, with $\mathfrak{K}_0(z,v)$ given in \eqref{F-intkerdelta}. Thus, given some $\gamma\in\Z^d$ let us consider the difference: $K_\delta-\tau_{-z_\gamma(\delta)}K_0\tau_{z_\gamma(\delta)}$ and its associated distribution kernel, considered as smooth distribution valued function on $\R^d$ and use Newton-Leibniz formula in the first variable to obtain:
	\begin{align} \nonumber
		\mathfrak{K}_\delta(z,\cdot)-\mathfrak{K}_0\big(z+\delta^\kappa\gamma,\cdot\big)
		&=\mathfrak{K}_0\big(z+\delta\,F(z),\cdot\big)-\mathfrak{K}_0\big(z+\delta^\kappa\gamma,\cdot\big)\\ \nonumber
		&=\int_0^1ds\,\Big(\big(\nabla_z\mathfrak{K}_0\big)\big(z+\delta^\kappa\gamma+s(\delta\,F(z)-\delta^\kappa\gamma),\cdot\big)\Big)\cdot\left(\delta\,F(z)-\delta^\kappa\gamma\right)\\ \label{F-kdelta-k0}
		&\cancel{\equiv} {\  \, \, =:\, }\delta^{\kappa}\,\big[\mathfrak{D}_1\mathfrak{K}_0\big](z,\cdot)\cdot\big(\delta^{1-\kappa}F(z)-\gamma\big)
	\end{align}
	with the last line giving the definition of $\big[\mathfrak{D}_1\mathfrak{K}_0\big](z,\cdot)$. We can then define the mapping $\boldsymbol{\Psi}_s[F]_\gamma^{(\delta)}:\R^d\ni\,z\mapsto\,z+\delta^\kappa\gamma+s(\delta\,F(z)-\delta^\kappa\gamma)\in\R^d$ and write that in the sense of tempered distributions:
\begin{align*}
	\big[\mathfrak{D}_1\mathfrak{K}_0\big]&=\int_0^1ds\,\big(\nabla_z\mathfrak{K}_0\big)\circ(\boldsymbol{\Psi}_s[F]_\gamma^{(\delta)},\bb1)=(2\pi)^{-d/2}\int_0^1ds\,\Big(\big((\bb1\otimes\mathcal{F}^-)(\nabla_za)\big)\circ(\boldsymbol{\Psi}_s[F]_\gamma^{(\delta)},\bb1)\Big)\\
	&=(2\pi)^{-d/2}(\bb1\otimes\mathcal{F}^-)\Big[\int_0^1ds\,\Big((\nabla_za)\big)\circ(\boldsymbol{\Psi}_s[F]_\gamma^{(\delta)},\bb1)\Big)\Big],
\end{align*}
with $\bb1:\Rd\rightarrow\Rd$ being the identity map, denoting by $(\Psi,\Phi)$ the map $\Rd\times\Rd\ni(x,y)\mapsto\big(\Psi(x),\Psi(y)\big)\in\Rd\times\Rd$ for any pair of maps $\Psi\in C^\infty(\Rd;\Rd)$ and $\Phi\in C^\infty(\Rd;\Rd)$. The above formula evidengtly implies that:
\beq\label{F-est-F-kdelta-k0}
(\bb1\otimes\mathcal{F})\big[\mathfrak{D}_1\mathfrak{K}_0\big]=(2\pi)^{-d/2}\Big[\int_0^1ds\,\Big((\nabla_za)\big)\circ(\boldsymbol{\Psi}_s[F]_\gamma^{(\delta)},\bb1)\Big)\Big]
\eeq
	In order to estimate the operator norm of the linear operator defined by this distribution kernel, we use our boundedness criterion \eqref{Cr-Bd} and notice that:
	\beq\label{F-est-kdelta-k0}\begin{split}
	&\partial^\alpha_z\partial^\beta_\xi(\bb1\otimes\mathcal{F})\big[\mathfrak{D}_1\mathfrak{K}_0\big]=(2\pi)^{-d/2}\Big[\int_0^1ds\,\Big((\nabla_z\partial^\alpha_z\partial^\beta_\xi\,a)\big)\circ(\boldsymbol{\Psi}_s[F]_\gamma^{(\delta)},\bb1)\Big)\Big]
	\end{split}\eeq
being bounded by $\nu_{|\alpha|+1,|\beta|}(a)$. Thus, if we can impose by some localization procedure, a bound uniform in $(z,\gamma)\in\R^d\times\Z^d$ for the factor $\delta^{1-\kappa}F(z)-\gamma$ and its $z$-derivatives then we may obtain a decaying factor $\delta^\kappa$ going to 0 with $\delta\geq0$. We are thus lead to consider the following partition of unity:
	\begin{itemize}
		\item We fix a function $g\in C^\infty_0\big(\R^d;[0,1]\big)$ such that: $\underset{\gamma\in\Z^d}{\sum}g(z-\gamma)^2=1,\ \forall z\in\R^d$.
		\item For any $\gamma\in\Z^d$ we define the cut-off function: $g[F_\delta]_\gamma(z):=g\big(\delta^{(1-\kappa)}F(z)-\gamma\big)$. 
		\item Given $\gamma\in \Z^d$ we denote by $V_\gamma$ the set of all $\gamma'\in \Z^d$ with the property that the support of  $g[F_\delta]_{\gamma^\prime}$ has a non-empty overlap with the support of $g[F_\delta]_\gamma$, including $\gamma'=\gamma$. Denote by $\mathfrak{n}_g\in\mathbb{N}\setminus\{0\}$ the cardinal of $V_\gamma$, notice that it is clearly independent of $\gamma$ and $\delta$ and that:
		\begin{align*}
		&\underset{\gamma\in\Z^d}{\sum}\big[g[F_\delta]_\gamma(z)\big]^2=1,\ \forall z\in\R^d,\\
		&z\in\supp\,g[F_\delta]_\gamma\ \Longrightarrow\ \exists L>0,\ \big|\delta^{(1-\kappa)}F(z)-\gamma\big|\,\leq\,L.
		\end{align*}
	\item Finally let us denote by $G[F_\delta]_\gamma$ the self-adjoint, bounded operator of multiplication with $g[F_\delta]_\gamma$ in $L^2(\R^d)$. Obviously $G[F_\delta]_\gamma=\Op^w\big(g[F_\delta]_\gamma\big)$ for $g[F_\delta]_\gamma\in S^0_{0,0}(\R^d\times\R^d)$ a symbol independent of the second variable.
	\end{itemize}
	
\paragraph{The quasi-rezolvent for $K_\delta$.} Let us fix any $\z\notin\sigma\big(K_0\big)$ and define:
\beq\label{DF-qRez-delta}
T_\gamma(\z;\delta):=\tau_{-z_\gamma(\delta)}\,\big(K_0-\z\bb1\big)^{-1}\tau_{z_\gamma(\delta)},\quad
\widetilde{T}(\z;\delta):=\underset{\gamma\in\Z^d}{\sum}G[F_\delta]_\gamma\,T_\gamma(\z;\delta)\,G[F_\delta]_\gamma.
\eeq

\begin{remark}\label{R-T-gamma}
Unitarity of translations and the functional calculus for self-adjoint operators imply that for any $\delta\in[0,1]$ we have the estimation: $$\big\|T_\gamma(\z;\delta)\big\|_{\mathbb{B}(L^2(\R^d))}\,=\,\big\|\big(K_0-\z\bb1\big)^{-1}\big\|_{\mathbb{B}(L^2(\R^d))}\,\leq\,\dist\big(\z,\sigma(K_0)\big)^{-1},\ \forall\gamma\in\Z^d.$$
\end{remark}
\begin{lemma}
For any $\delta\in[0,1]$ the series in \eqref{DF-qRez-delta} is convergent in the strong operator topology and we have the estimation $$\big\|\widetilde{T}(\z;\delta)\big\|_{\mathbb{B}(L^2(\R^d))}\,\leq\,\sqrt{(\mathfrak{n}_g+1)/2}\big\|\big(K_0-\z\bb1\big)^{-1}\big\|_{\mathbb{B}(L^2(\R^d))}\,\leq\,\sqrt{(\mathfrak{n}_g+1)/2}\dist\big(\z,\sigma(K_0)\big)^{-1}.$$
\end{lemma}
\begin{proof} (\textit{For the convenience of the reader we reproduce here our proof of Lemma 2.4 in \cite{CP3}})
	For $\psi\in L^2(\R^d)$, let us consider any $M\in\mathbb{N}$ let us define:
	\beq\label{F-part-sum}
	\widetilde{T}(\z;\delta)^{(M)}:=\underset{|\gamma|\leq M}{\sum}G[F_\delta]_\gamma\,T_\gamma(\z;\delta)\,G[F_\delta]_\gamma
	\eeq
	and compute:
	\begin{align*}
		&\sum_{|\gamma|\leq M}\sum_{\gamma'\in V_\gamma}\langle G[F_\delta]_\gamma\,T_\gamma(\z;\delta)\,G[F_\delta]_\gamma \; \psi\,,\, G[F_\delta]_{\gamma^\prime}\,T_{\gamma^\prime}(\z;\delta)\,G[F_\delta]_{\gamma^\prime}\;  \psi \rangle \leq\\
		&\leq\, \frac{\nu +1}{2}\sum_{|\gamma|\leq M}\Vert T_{\gamma^\prime}(\z;\delta)\,G[F_\delta]_{\gamma^\prime}\; \psi\Vert ^2\leq \big\|\big(K_0-\z\bb1\big)^{-1}\big\|^2_{\mathbb{B}(L^2(\R^d))}\frac{\mathfrak{n}_g+1}{2}\; \underset{|\gamma|\leq M}{\sum}\int_{\R^d}dz\,\big[g[F_\delta]_\gamma(z)\big]^2|\psi(z)|^2\; dx\\
		&\leq\,\big\|\big(K_0-\z\bb1\big)^{-1}\big\|^2_{\mathbb{B}(L^2(\R^d))}\; \frac{\mathfrak{n}_g+1}{2}\; \Vert \psi\Vert^2,
	\end{align*}
	where in the last equality we used the quadratic partition of unity identity in the definition of $g\in C^\infty_0\big(\R^d;[0,1]\big)$. The convergence and the estimation in the Lemma are then evident.
\end{proof}

\begin{proposition}\label{P-main-est}
With the above notations and hypothesis, we have the estimation:
$$
\Big\|\big(K_\delta-\z\bb1\big)\,\widetilde{T}(\z;\delta)\,-\,\bb1\Big\|_{\mathbb{B}(L^2(\R^d))}\,\leq\,C(a,F)\,\delta^{1/2}\,\big[\dist\big(\z,\sigma(K_0)\big)\big]^{-1}.
$$
\end{proposition}
\begin{proof}
For any $M\in\mathbb{N}$ let $\widetilde{T}(\z;\delta)^{(M)}$ as in \eqref{F-part-sum} be the partial sum approaching $\widetilde{T}(\z;\delta)$ in the strong operator topology and let us consider the product:
\beq\label{Step-0}
\big(K_\delta-\z\bb1\big)\,\widetilde{T}(\z;\delta)^{(M)}=\underset{|\gamma|\leq M}{\sum}\big(K_\delta-\z\bb1\big)\,G[F_\delta]_\gamma\,T_\gamma(\z;\delta)\,G[F_\delta]_\gamma.
\eeq

We notice that for any $\gamma\in\Z^d$ we can write that:
\begin{align}\label{Step-1}
\big(K_\delta-\z\bb1\big)\,G[F_\delta]_\gamma\,T_\gamma(\z;\delta)\,G[F_\delta]_\gamma&=\tau_{-z_\gamma(\delta)}\,\big(K_0-\z\bb1\big)^{-1}\tau_{z_\gamma(\delta)}\,G[F_\delta]_\gamma\,T_\gamma(\z;\delta)\,G[F_\delta]_\gamma\,+\\ \nonumber
&\hspace*{0.5cm}+\,\Big[\big(K_\delta-\z\bb1\big)\,-\,\tau_{-z_\gamma(\delta)}\,\big(K_0-\z\bb1\big)^{-1}\tau_{z_\gamma(\delta)}\Big]\,G[F_\delta]_\gamma\,T_\gamma(\z;\delta)\,G[F_\delta]_\gamma\\ \nonumber
&=\,\big[g[F_\delta]_\gamma\big]^2\,\bb1\,+\,\Big[\tau_{-z_\gamma(\delta)}\,\big(K_0-\z\bb1\big)^{-1}\tau_{z_\gamma(\delta)}\,,\,G[F_\delta]_\gamma\Big]T_\gamma(\z;\delta)\,G[F_\delta]_\gamma\,+\\ \nonumber
&\hspace*{0.5cm}+\,\Big[\big(K_\delta-\z\bb1\big)\,-\,\tau_{-z_\gamma(\delta)}\,\big(K_0-\z\bb1\big)^{-1}\tau_{z_\gamma(\delta)}\Big]\,G[F_\delta]_\gamma\,T_\gamma(\z;\delta)\,G[F_\delta]_\gamma.
\end{align}
\begin{lemma}\label{L-1} For any $\gamma\in\Z^d$ we have the estimation:
$$\Big\|\Big[\big(K_\delta-\z\bb1\big)\,-\,\tau_{-z_\gamma(\delta)}\,\big(K_0-\z\bb1\big)^{-1}\tau_{z_\gamma(\delta)}\Big]\,G[F_\delta]_\gamma\Big\|_{\mathbb{B}(L^2(\R^d))}\,\leq\,C(a,F)\delta^\kappa.
$$
\end{lemma}
\begin{proof}
We consider thr bounded operator:
\beq
H_\delta:=\Big[\big(K_\delta-\z\bb1\big)\,-\,\tau_{-z_\gamma(\delta)}\,\big(K_0-\z\bb1\big)^{-1}\tau_{z_\gamma(\delta)}\Big]\,G[F_\delta]_\gamma
\eeq 
appearing in the statement of the Lemma and compute its distribution kernel:
\begin{align}
\mathfrak{K}[H_\delta]=\Big[\mathfrak{K}\big[a[F]_\delta\big]-\Big(\mathfrak{K}[a]\circ(\tau_{z_\gamma(\delta)}\otimes\tau_{z_\gamma(\delta)})\Big)\Big]\big(1\otimes g[F_\delta]_\gamma\big).
\end{align}
We shall prefer to work with the modified kernels: 
\begin{align} \nonumber
\check{\mathfrak{K}}_\delta=\mathfrak{K}[H_\delta]\circ\Upsilon^{-1}&=\Big[\Big(\mathfrak{K}\big[a[F]_\delta\big]-\big(\mathfrak{K}[a]\circ(\tau_{z_\gamma(\delta)}\otimes\tau_{z_\gamma(\delta)})\big)\Big)\circ\Upsilon^{-1}\Big]\big[\big(1\otimes g[F_\delta]_\gamma\big)\circ\Upsilon^{-1}\big]\\ \nonumber
&=\Big(\mathfrak{K}_\delta-\big(\mathfrak{K}_0\circ(\tau_{z_\gamma(\delta)}\otimes1)\big)\Big)\big[\big(1\otimes g[F_\delta]_\gamma\big)\circ\Upsilon^{-1}\big]\\ \label{F-1}
&=\Big[\Big(\mathfrak{K}_\delta-\big(\mathfrak{K}_0\circ(\tau_{z_\gamma(\delta)}\otimes1)\big)\Big)(1\otimes\mathfrak{s}_N)\Big]\Big[\big[\big(1\otimes g[F_\delta]_\gamma\big)\circ\Upsilon^{-1}\big](1\otimes\mathfrak{s}_{-N})\Big].
\end{align}
with the last line valid for any $N\in\mathbb{N}$ and 
the first factor above being bounded for any $N\in\mathbb{N}$ due to the arguments using \eqref{F-est-kdelta-k0}. In fact, by \eqref{F-kdelta-k0} we can write:
\begin{align*}\nonumber
(\bb1\otimes\mathcal{F})\check{\mathfrak{K}}_\delta=\delta^{\kappa}\,&\Big[\underset{1\leq j\leq d}{\sum}\Big((\bb1\otimes\mathcal{F})\big[\mathfrak{D}_1\mathfrak{K}_0\big]_j\big)\star\Big((\bb1\otimes\mathcal{F})\big[\big(\delta^{1-\kappa}F_j-\gamma_j\big)\otimes\mathfrak{s}_{N}\big]\Big)\Big]\star\,\\
&\,\star\Big[(\bb1\otimes\mathcal{F})\Big(\big[\big(1\otimes g[F_\delta]_\gamma\big)\circ\Upsilon^{-1}\big](1\otimes\mathfrak{s}_{-N})\Big)\Big]
\end{align*} 
Concerning the second factor above we notice that:
\beq
\big[\big(1\otimes g[F_\delta]_\gamma\big)\circ\Upsilon^{-1}\big](z,v)=g\big(\delta^{(1-\kappa)}\,F(z-v/2)-\gamma\big)
\eeq 
and using the compactness of the support of the cut-off function $g$ we deduce that on the support of the function $(\bb1\otimes\mathcal{F})\Big(\big[\big(1\otimes g[F_\delta]_\gamma\big)\circ\Upsilon^{-1}\big](1\otimes\mathfrak{s}_{-N})\Big)$ there exists some $L>0$, depending only on the diameter of the support of $g$ such that:
\begin{align*}
L\,&\geq\big|\delta^{(1-\kappa)}F(z-v/2)-\gamma\big|=\Big|\delta^{(1-\kappa)}\Big(F(z)-\int_0^1ds\,\big[(v/2)\cdot\big(\nabla\,F\big)(z-sv/2)\big]\Big)-\gamma\Big|\geq\\
&\geq\Big|\big|\delta^{(1-\kappa)}F(z)-\gamma\big|\,-\,\Big|\delta^{(1-\kappa)}\int_0^1ds\,\big[(v/2)\cdot\big(\nabla\,F\big)(z-sv/2)\big]\Big|\,\Big|
\end{align*}
and thus we have the inequality:
\beq\nonumber
\big|\delta^{(1-\kappa)}F(z)-\gamma\big|\,\leq\,L+\delta^{(1-\kappa)}\Big|\int_0^1ds\,\big[(v/2)\cdot\big(\nabla\,F\big)(z-sv/2)\big]\Big|\,\leq\,L+\delta^{(1-\kappa)}\big((1/2)\|\nabla\,F\|_\infty\big)<v>.
\eeq

Moreover, one easily notices that the function $\Big(\big[\big(1\otimes g[F_\delta]_\gamma\big)\circ\Upsilon^{-1}\big](1\otimes\mathfrak{s}_{-N})\Big)$ is of class $BC^\infty(\R^d\times\R^d)$ having rapid decay in the second variable, with uniform bounds with respect to $\delta\in[0,1]$, so that its partial Fourier transform  $(\bb1\otimes\mathcal{F})\Big(\big[\big(1\otimes g[F_\delta]_\gamma\big)\circ\Upsilon^{-1}\big](1\otimes\mathfrak{s}_{-N})\Big)$ is a function of class  $BC^\infty(\R^d\times\R^d)=S^0_{0,0}(\R^d\times\R^d)$, uniformly with respect to $\delta\in[0,1]$.

Recalling now our boundedness criterion \eqref{Cr-Bd}:
\begin{align}\nonumber
&\Big\|\Big[\big(K_\delta-\z\bb1\big)\,-\,\tau_{-z_\gamma(\delta)}\,\big(K_0-\z\bb1\big)^{-1}\tau_{z_\gamma(\delta)}\Big]\,g[F_\delta]_\gamma\Big\|_{\mathbb{B}(L^2(\R^d))}\leq\nu_{3d+4,3d+4}\big((\bb1\otimes\mathcal{F})\check{\mathfrak{K}}_\delta\big),
\end{align}
the conclusion of the Lemma follows.
\end{proof}

\begin{lemma} For any $\gamma\in\Z^d$ we have the estimation:
	$$\Big\|\Big[\tau_{-z_\gamma(\delta)}\,\big(K_0-\z\bb1\big)^{-1}\tau_{z_\gamma(\delta)}\,,\,G[F_\delta]_\gamma\Big]\Big\|_{\mathbb{B}(L^2(\R^d))}
	\,\leq\,C(a,F)\,\delta^{(1-\kappa)}.
	$$
\end{lemma}
\begin{proof}
In a similar way with the proof of our previous Lemma \ref{L-1} we consider the linear operator:
\beq\begin{split}
&\Big[\tau_{-z_\gamma(\delta)}\,\big(K_0-\z\bb1\big)^{-1}\tau_{z_\gamma(\delta)}\,,\,G[F_\delta]_\gamma\Big]=\\
&\hspace*{2cm}=\tau_{-z_\gamma(\delta)}\,\big(K_0-\z\bb1\big)^{-1}\tau_{z_\gamma(\delta)}\,,\,G[F_\delta]_\gamma\,-\,\,G[F_\delta]_\gamma\,\tau_{-z_\gamma(\delta)}\,\big(K_0-\z\bb1\big)^{-1}\tau_{z_\gamma(\delta)}
\end{split}\eeq
and its distribution kernel:
\beq
\mathfrak{K}_{C,\delta}\,:=\,\Big(\mathfrak{K}[a]\circ(\tau_{z_\gamma(\delta)}\otimes\tau_{z_\gamma(\delta)})\Big)\Big[\big(1\otimes g[F_\delta]_\gamma\big)\,-\,\big(g[F_\delta]_\gamma\otimes 1\big)\Big]
\eeq
with the modified form:
\beq
\widetilde{\mathfrak{K}}_{C,\delta}:=\mathfrak{K}_{C,\delta}\circ\Upsilon^{-1}=\big[\mathfrak{K}_0\circ(\tau_{z_\gamma(\delta)}\otimes\bb1)\big]\Big[\Big(\big(1\otimes g[F_\delta]_\gamma\big)\circ\Upsilon^{-1}\Big)\,-\,\Big(\big(g[F_\delta]_\gamma\otimes 1\big)\circ\Upsilon^{-1}\Big)\Big].
\eeq
Let us analyse the smooth function in the second factor above:
\begin{align}\nonumber
&\Big[\Big(\big(1\otimes g[F_\delta]_\gamma\big)\circ\Upsilon^{-1}\Big)\,-\,\Big(\big(g[F_\delta]_\gamma\otimes 1\big)\circ\Upsilon^{-1}\Big)\Big](z,v)=\\ \nonumber
&=g\big(\delta^{(1-\kappa)}F(z-v/2)-\gamma\big)\,-\,g\big(\delta^{(1-\kappa)}F(z+v/2)-\gamma\big)=\\ \nonumber
&=-\int_{0}^{1}\hspace*{-0,3cm}ds\,\big((\nabla\,g)\big(\delta^{(1-\kappa)}F(z-v/2)-\gamma+s\delta^{(1-\kappa)}\big(F(z+v/2)-F(z-v/2)\big)\cdot\\ \nonumber
&\hspace*{5cm}\cdot(\delta^{(1-\kappa)})\big(F(z+v/2)-F(z-v/2)\big)=\\
&=-\delta^{(1-\kappa)}\underset{1\leq j,k\leq d}{\sum}\int_{0}^{1}\hspace*{-0,3cm}ds\int_{-1/2}^{1/2}\hspace*{-0,3cm}dt\,v_k\,\partial_kF_j(z+sv)\,\times\\ \nonumber
&\hspace*{5cm}\times\,\big(\partial_j\,g\big)\big(\delta^{(1-\kappa)}F(z-v/2)-\gamma+s\delta^{(1-\kappa)}\big(F(z+v/2)-F(z-v/2)\big)
\end{align}
and our usual boundedness criterion clearly implies the conclusion of the Lemma.
\end{proof}
Putting together \eqref{Step-1}, Remark \ref{R-T-gamma} and the above two lemmas, and optimizing the estimation by taking $\kappa=1-\kappa=1/2$ we conclude that:
\beq\label{Step-2}\begin{split}
	&\big(K_\delta-\z\bb1\big)\,G[F_\delta]_\gamma\,T_\gamma(\z;\delta)\,G[F_\delta]_\gamma-\big[g[F_\delta]_\gamma\big]^2\,\bb1\,=\,X^{(\delta)}_\gamma\,G[F_\delta]_\gamma,\\
	&\Big\|X^{(\delta)}_\gamma\Big\|_{\mathbb{B}(L^2(\R^d))}\,\leq\,C(a,F)\,\delta^{1/2}\,\big(\dist\big(\z,\sigma(K_0)\big)^{-1}.
\end{split}\eeq

Finally we have to use the fact that $\underset{\gamma\in\Z^d}{\sum}g(x-\gamma)^2=1$ and $\underset{\gamma\in\Z^d}{\sum}g(x-\gamma)\in[0,\mathfrak{n}_g]$, both series being locally finite, so that the finite sums in \eqref{Step-0} are convergent and summing up over $\gamma\in\Z^d$ using the estimation \eqref{Step-2} clearly implies the conclusion of the Proposition.
\end{proof}

\paragraph{End of the proof of Theorem \ref{T-A}.} If $\dist\big(\z,\sigma\big(\Op(a)\big)\big)\geq\,C\delta^{1/2}$
the conclusion of Proposition \ref{P-main-est} implies that $\z\notin\sigma\big(K_\delta\big)$.

Finally, replacing $\widetilde{T}(\z;\delta)$ in \eqref{DF-qRez-delta} by:
\beq\label{DF-qRez-delta-inv}
\widetilde{S}(\z;\delta):=\underset{\gamma\in\Z^d}{\sum}G[F_\delta]_\gamma\,\tau_{z_\gamma(\delta)}\,\big(K_\delta-\z\bb1\big)^{-1}\tau_{-z_\gamma(\delta)}\,G[F_\delta]_\gamma
\eeq
all the arguments above can still be applied in order to obtain the following estimation similar to the conclusion of Proposition \ref{P-main-est}:
\beq
\Big\|\big(K_0-\z\bb1\big)\,\widetilde{S}(\z;\delta)\,-\,\bb1\Big\|_{\mathbb{B}(L^2(\R^d))}\,\leq\,C(a,F)\,\delta^{1/2}\,\big(\dist\big(\z,\sigma(K_\delta)\big)^{-1}.
\eeq
It follows that if $\dist\big(\z,\sigma\big(\Op^w\big(a[F]_\delta\big)\big)\big)\geq\,C\delta^{1/2}$
then $\z\notin\sigma\big(K_0\big)$ and the Theorem is proven.

\section{Proof of Theorem \ref{T-B}}

In this case, we shall no longer estimate norms, but rather quadratic forms.
The main idea is to replace the perturbation $x\mapsto x+\delta\,F(x)$ with a similar one in a new variable $u\in\R^d$, namely $x\mapsto x+\delta\,F(u)$ and use the unitarity of translations in estimating the modified quadratic form.
In order to control the distance between the new variable $u$ and $z:=(x+y)/2$ we shall use a scaled weight function $\mathfrak{W}_{\kappa}(z-u)$ as in \cite{CP3} (see \eqref{DF-W}). {\  We shall treat only the case $\mathcal{E}_+(\delta)-\mathcal{E}_+(0)$, the other one, i.e. $\mathcal{E}_-(\delta)-\mathcal{E}_-(0)$ following by a quite similar argument.}

We intend to estimate the difference $\mathcal{E}_+(\delta)-\mathcal{E}_+(0)$ for $\delta>0$ small enough, and begin by making explicit the defining formula \eqref{F-spedgedelta}: 
\beq\begin{split}\label{F-Eplus-delta}
	\mathcal{E}_+(\delta)&=\underset{\|\phi\|_{L^2(\R^d)}=1}{\sup}\big(\phi\,,\,\Op\big(a[F]_\delta\big)\,\phi\big)_{L^2(\R^d)}=\underset{\|\phi\|_{L^2(\R^d)}=1}{\sup}\big\langle\mathfrak{K}\big[a[F]_\delta\big]\,,\,\overline{\phi}\otimes\phi\big\rangle_{\mathscr{S}(\R^d\times\R^d)}\\
	&=\underset{\|\phi\|_{L^2(\R^d)}=1}{\sup}\big\langle\mathfrak{K}_\delta\,,\,(\overline{\phi}\otimes\phi)\circ\Upsilon^{-1}\big\rangle_{\mathscr{S}(\R^d\times\R^d)}.
\end{split}\eeq

\paragraph{The weight function.}  Let us consider the functions:
\beq\label{DF-W}
\mW(z):=(4\pi)^{-d/2}\, e^{-\frac{|z|^2}{4}},\qquad\mW_\kappa(z):=\kappa^{d/2}\mW(\kappa\,z),\ \forall\kappa\in(0,1]
\eeq
and the following identity:
\beq
2^{-1}\big( |w+v/2|^2 + |w-v/2|^2\big)=|w|^2+|v|^2/4,\qquad\forall(w,v)\in\R^d\times\R^d.
\eeq
We deduce that:
\beq\begin{split}\label{F-weight-dec}
	&\int_{\R^d}dz\,\mW_\kappa(z)\ =\ 1,\quad\forall\kappa\in(0,1],\\
	&\mW_\kappa(z-u)\,=\,\big((\kappa/4\pi)^{-d/2}\mW_\kappa(v)\big)^{-1/4}\mW_\kappa(z-u+v/2)^{1/2}\mW_\kappa(z-u-v/2)^{1/2},\\ &\hspace*{9cm}\forall(z,u,v,\kappa)\in[\R^d]^3\times(0,1].
\end{split}\eeq

Our strategy is to replace in formula \eqref{F-Eplus-delta} the distribution:
\beq
\mathfrak{K}_\delta\,=\,\Big(\int_{\R^d}du\,\big((\tau_u\mW_\kappa)\otimes1\big)\Big)\,\mathfrak{K}_\delta,
\eeq
with the distribution:
\beq
\mW_\kappa[\mathfrak{K}_\delta]:=\int_{\R^d}du\,\big((\tau_u\mW_\kappa)\otimes1\big)\Big(\big(\tau_{-\delta\,F(u)}\otimes\bb1\big)\mathfrak{K}_0\Big)
\eeq
and estimate:
\beq\label{DF-def-spedge}
\widetilde{\mathcal{E}_+}(\kappa,\delta):=\underset{\|\phi\|_{L^2}=1}{\sup}\big\langle\mW_\kappa[\mathfrak{K}_\delta]\,,\,(\overline{\phi}\otimes\phi)\circ\Upsilon^{-1}\big\rangle_{\mathscr{S}(\R^d\times\R^d)}.
\eeq

\begin{proposition}\label{P-B1}
	With the above notations and hypothesis, for any $\phi\in\mathscr{S}(\R^d)$ there exists some $C(a,F)>0$ such that we have the estimation:
	$$
	\big\langle\mW_\kappa[\mathfrak{K}_\delta]\,,\,(\overline{\phi}\otimes\phi)\circ\Upsilon^{-1}\big\rangle_{\mathscr{S}(\R^d\times\R^d)}=\big(\phi\,,\,\Op(a)\,\phi\big)_{L^2(\R^d)}\,+\,C(a,F)\,\kappa^2\,\|\phi\|_{L^2}^2,\quad\forall\phi\in L^2(\R^d).
	$$
\end{proposition}
\begin{proof}
	Starting from \eqref{DF-def-spedge}, we have to estimate the following iterated integrals:
	\begin{align}
		\left|\int_{\R^d}du\int_{\R^d}dz\,\mW_\kappa(z-u)\int_{\R^d}dv\,\overline{\phi(z+v/2)}\phi(z-v/2)\,\mathfrak{K}_0\big(z+\delta F(u),v\big)\right|.
	\end{align}
	
	We shall use the rapid decay in $v\in\R^d$ of the kernel $\mathfrak{K}_0\big(z+\delta F(u),v\big)$ by breaking the integral in $v\in\R^d$ in a bounded region and its complementary. In fact we shall choose some function $\chi\in C^\infty_0(\R^d)$, taking values in $[0,1]$, having support in the ball $|v|\leq R$ and being equal to 1 on the ball $|v|\leq r$, for some strictly positive $r<R$.
	
	Let us first estimate the integral on the unbounded region, for any $\kappa\in(0,1)$ and $N\in\mathbb{N}$:
	\begin{align*}
		\left|\int_{\R^d}du\int_{\R^d}dz\,\mW_\kappa(z-u)\int_{\R^d}dv\,\overline{\phi(z+v/2)}\phi(z-v/2)\,\mathfrak{K}_0\big(z+\delta F(u),v\big)<v>^N<v>^{-N}[1-\chi(\kappa\,v)]\right|\\
		\leq\,\kappa^{N}\Big(r^{-N}\underset{z\in\R^d}{\sup}\,\underset{|v|\geq r}{\sup}<v>^N\big|\mathfrak{K}_0\big(z,v\big)\big|\Big)\,\|\phi\|_{L^2(\R^d)}^2\,\leq\,C_r(a)\,\kappa^N\,\|\phi\|_{L^2(\R^d)}^2.
	\end{align*}
	
	On the support of $\chi$ we shall use the second formula in \eqref{F-weight-dec} in order to write that:
	\begin{align}
		&\int_{\R^d}du\int_{\R^d}dz\,\mW_\kappa(z-u)\int_{\R^d}dv\,\overline{\phi(z+v/2)}\phi(z-v/2)\,\mathfrak{K}_0\big(z+\delta F(u),v\big)\,\chi(\kappa\,v)=\\ \nonumber
		&\hspace*{4cm}=\Big(\widetilde{\phi}\,,\,\big(\bb1_{\R^d_u}\otimes\tau_{-\delta F(u)}\big)\big(\bb1_{\R^d_u}\otimes\Op^w(a_\kappa)\big)\big(\bb1_{\R^d_u}\otimes\tau_{\delta F(u)}\big)\widetilde{\phi}\Big)_{L^2(\R^d;L^2(\R^d))}
	\end{align}
	with:
	\begin{align*}
		\big(\mathfrak{K}[a_\kappa]\circ\Upsilon\big)(z,v):&=\big((\kappa/4\pi)^{-d/2}\mW_\kappa(v)\big)^{-1/4}\mathfrak{K}_0(z,v)\,\chi(\kappa\,v)=e^{|\kappa\,v|^2/16}\,\chi(\kappa\,v)\,\mathfrak{K}_0(z,v),\\ \nonumber
		&=\mathfrak{K}_0(z,v)\,\chi(\kappa\,v)\,+\,\kappa^2\int_0^1ds\,(|v|^2/16)\,e^{s|\kappa\,v|^2/16}\,\chi(\kappa\,v)\,\mathfrak{K}_0(z,v),\\
		\widetilde{\phi}:&=\big(\tau_{-u}\mW_\kappa\big)\,\phi\in\,L^2\big(\R^d_u;L^2(\R^d_z)\big).
	\end{align*}
	We notice that we have a unitary map 
	\beq
	L^2(\R^d)\ni\phi\,\mapsto\,\widetilde{\phi}:=\big(\tau_{-u}\mW_\kappa\big)\,\phi\in\,L^2\big(\R^d_u;L^2(\R^d_z)\big)
	\eeq
	and the equality (taking into account the unitarity of translations):
	\begin{align*}
		\Big(\widetilde{\phi}\,,\,\big(\bb1_{\R^d_u}\otimes\tau_{-\delta F(u)}\big)\big(\bb1_{\R^d_u}\otimes\Op^w(a_\kappa)\big)\big(\bb1_{\R^d_u}\otimes\tau_{\delta F(u)}\big)\widetilde{\phi}\Big)_{L^2(\R^d;L^2(\R^d))}=\Big(\phi\,,\,\Op^w(a_\kappa)\phi\Big)_{L^2(\R^d)}.
	\end{align*}
	Thus if we change the Hilbert space $L^2(\R^d)$ with the Hilbert space $L^2\big(\R^d;L^2(\R^d)\big)$ via the above unitary map we may conclude that:
	\begin{align*}
		&\int_{\R^d}du\int_{\R^d}dz\,\mW_\kappa(z-u)\int_{\R^d}dv\,\overline{\phi(z+v/2)}\phi(z-v/2)\,\mathfrak{K}_0\big(z+\delta F(u),v\big)\,\chi(\kappa\,v)=\\
		&\hspace*{0.2cm}=\Big(\phi\,,\,\Op^w(a_\kappa)\phi\Big)_{L^2(\R^d)}\ =\ \Big(\phi\,,\,\Op^w(a_\chi)\phi\Big)_{L^2(\R^d)}\,+\\ \nonumber
		&\hspace*{0.6cm}+\kappa^2\int_{\R^d}dz\int_{\R^d}dv\,\overline{\phi(z+v/2)}\phi(z-v/2)\Big(\int_0^1ds\,e^{s|\kappa\,v|^2/16}\Big)\,(|v|^2/16)\,\mathfrak{K}_0\big(z+\delta F(u),v\big)\,\chi(\kappa\,v)
	\end{align*}
	where we have put into evidence the symbol $a_\chi\in S^0_{0,0}(\R^d\times\R^d)$ associated to the integral kernel $\mathfrak{K}[a](x,y)\chi\big(\kappa(x-y)\big)$. Then we may control the factor $(|v|^2/16)$ using the rapid decay of $\mathfrak{K}_0$ with respect to the variable $v\in\R^d$ and write that $\exp\big(s|\kappa\,v|^2/16\big)\chi(\kappa\,v)\leq\,\exp(R^2/16)$.
	Finally we use once again the estimation on the support of $1-\chi$:
	\begin{align*}
		&\big(\phi\,,\,\Op(a_\chi)\,\phi\big)_{L^2(\R^d)}\,=\\
		&\hspace*{0,5cm}=\,\big(\phi\,,\,\Op(a)\,\phi\big)_{L^2(\R^d)}\,-\,\int_{\R^d}dz\,\int_{\R^d}dv\,\overline{\phi(z+v/2)}\phi(z-v/2)\,\mathfrak{K}_0\big(z,v\big)[1-\chi(\kappa\,v)]\\
		&\Big|\big(\phi\,,\,\Op(a)\,\phi\big)_{L^2(\R^d)}\,-\,\int_{\R^d}dz\,\int_{\R^d}dv\,\overline{\phi(z+v/2)}\phi(z-v/2)\,\mathfrak{K}_0\big(z,v\big)[1-\chi(\kappa\,v)]\Big|\,\leq\\
		&\hspace*{2cm}\leq\,\kappa^{N}\Big(r^{-N}\underset{z\in\R^d}{\sup}\,\underset{|v|\geq r}{\sup}<v>^N\big|\mathfrak{K}_0\big(z,v\big)\big|\Big)\,\|\phi\|_{L^2(\R^d)}^2\,\leq\,C(a)\,\kappa^N\,\|\phi\|_{L^2(\R^d)}^2.
	\end{align*}
\end{proof}
This Proposition clearly implies the estimation:
\beq\label{E-1}\ 
\big|\widetilde{\mathcal{E}_+}(\delta)\,-\,\mathcal{E}_+(0)\big|\,=\,\mathcal{O}(\kappa^2).
\eeq
\begin{proposition}
There exists $C(a,F)>0$ such that for any $(\kappa,\theta)\in(0,1)\times(0,1)$ and for any $\phi\in\,L^2(\R^d)$:
\beq\begin{split}
&\big\langle\mathfrak{K}_\delta\,,\,(\overline{\phi}\otimes\phi)\circ\Upsilon^{-1}\big\rangle_{\mathscr{S}(\R^d\times\R^d)}\,-\,\big\langle\mW[\mathfrak{K}_\delta]\,,\,(\overline{\phi}\otimes\phi)\circ\Upsilon^{-1}\big\rangle_{\mathscr{S}(\R^d\times\R^d)}=\\
&\hspace*{7cm}=\,C(a,F)\,\|\phi\|_{L^2(\R^d)}^2\,\big(\delta/\theta+\delta\kappa^{-2}\theta^{1+\mu}+\delta^2\kappa^{-2}\big).
\end{split}\eeq
\end{proposition}
\begin{proof}
\begin{align*}
	&\big\langle\mathfrak{K}_\delta\,,\,(\overline{\phi}\otimes\phi)\circ\Upsilon^{-1}\big\rangle_{\mathscr{S}(\R^d\times\R^d)}\,-\,\big\langle\mW[\mathfrak{K}_\delta]\,,\,(\overline{\phi}\otimes\phi)\circ\Upsilon^{-1}\big\rangle_{\mathscr{S}(\R^d\times\R^d)}=\\
	&\hspace*{0,2cm}=\int_{\R^d}du\int_{\R^d}dz\,\mW_\kappa(z-u)\int_{\R^d}dv\,\overline{\phi(z+v/2)}\phi(z-v/2)\,\left[\mathfrak{K}_0\big(z+\delta F(z),v\big)\,-\,\mathfrak{K}_0\big(z+\delta F(u),v\big)\right]\\
	&\hspace*{0,2cm}=-\delta\hspace*{-0,2cm}\int_{\R^d}\hspace*{-0,2cm}du\int_{\R^d}\hspace*{-0,2cm}dz\,\mW_\kappa(z-u)\hspace*{-0,2cm}\int_{\R^d}\hspace*{-0,2cm}dv\,\overline{\phi(z+v/2)}\phi(z-v/2)\,\big(\nabla_z\mathfrak{K}_0\big)\big(z+\delta F(z)+s\delta(F(u)-F(z)),v\big)\times\\
	&\hspace*{1cm}\times\,\Big[\int_0^1ds\,\big((z-u)\cdot\nabla F\big)\big(z+s(u-z)\big)\Big].
\end{align*}
We shall need a second cut-off, this time on the perturbing field $F\in C^\infty_1(\R^d)$. Let us consider the same function $\chi\in C^\infty_0(\R^d)$ as in the proof above and the weighted one $\chi_\theta(z):=\chi(\theta\,z)$ for some cut-off parameter $\theta\in(0,1]$. Then we define:
\beq
F_\theta:=\chi_\theta\,F,\qquad F_\theta^\bot:=(1-\chi_\theta)\,F
\eeq
and the corresponding integral kernels $\mathfrak{K}_\delta^\circ$ and $\mathfrak{W}[\mathfrak{K}_\delta]^\circ$ with $F$ replaced by $F_\theta$ and respectively $\mathfrak{K}_\delta^\bot$ and $\mathfrak{W}[\mathfrak{K}_\delta]^\bot$ with $F$ replaced by $F_\theta^\bot$.

We have the evident estimations:
\begin{align*}
&\big\langle\mathfrak{K}_\delta^\circ\,,\,(\overline{\phi}\otimes\phi)\circ\Upsilon^{-1}\big\rangle_{\mathscr{S}(\R^d\times\R^d)}-\big\langle\mathfrak{K}_0\,,\,(\overline{\phi}\otimes\phi)\circ\Upsilon^{-1}\big\rangle_{\mathscr{S}(\R^d\times\R^d)}=\\
&=\int_{\R^d}dz\int_{\R^d}dv\,\overline{\phi(z+v/2)}\phi(z-v/2)\,\Big[\mathfrak{K}_0\big(z+\delta F_\theta(z),v\big)-\mathfrak{K}_0\big(z,v\big)\Big]=\\
&=\delta\int_{\R^d}dz\int_{\R^d}dv\,\overline{\phi(z+v/2)}\phi(z-v/2)\int_0^1ds\,\big[F_\theta(z)\cdot\big(\partial_z\mathfrak{K}_0\big)\big(z+s\delta\,F_\theta(z),v\big)\big],
\end{align*}
\begin{align*}
&\big\langle\mathfrak{W}[\mathfrak{K}_\delta]^\circ\,,\,(\overline{\phi}\otimes\phi)\circ\Upsilon^{-1}\big\rangle_{\mathscr{S}(\R^d\times\R^d)}-\big\langle\mathfrak{K}_0\,,\,(\overline{\phi}\otimes\phi)\circ\Upsilon^{-1}\big\rangle_{\mathscr{S}(\R^d\times\R^d)}=\\
&=\Big|\int_{\R^d}du\int_{\R^d}dz\,\mW_\kappa(z-u)\int_{\R^d}dv\,\overline{\phi(z+v/2)}\phi(z-v/2)\,\Big[\mathfrak{K}_0\big(z+\delta F_\theta(u),v\big)-\mathfrak{K}_0\big(z,v\big)\Big]\Big|=\\
&=\delta\int_{\R^d}dz\int_{\R^d}dv\,\overline{\phi(z+v/2)}\phi(z-v/2)\,\Big[\int_{\R^d}du\,\mW_\kappa(z-u)\,\times\\
&\hspace{7cm}\times\,\Big(\int_0^1ds\,\big[F_\theta(u)\cdot\big(\partial_z\mathfrak{K}_0\big)\big(z+s\delta\,F_\theta(u),v\big)\big]\Big)\Big]
\end{align*}
{\  We estimate the above two differences in the next two Lemmas.}

\begin{lemma}
The symbol $a'_{\delta,\theta}(z,\eta)$ associated to the kernel $\mathfrak{K}_{\delta,\theta}'(z,v):=\int_0^1ds\,\big[F_\theta(z)\cdot\big(\partial_z\mathfrak{K}_0\big)\big(z+s\delta\,F_\theta(z),v\big)\big]$ belongs to $S^0_{0,0}(\R^d\times\R^d)$ with seminorms bounded by $C\theta^{-1}$ uniformly for $(\delta,\theta)\in(0,1]^2$.
\end{lemma}
\begin{proof}
We can write that: \begin{align*}
a'_{\delta,\theta}(z,\eta)&=(2\pi)^d\int_{\R^d}dv\,e^{-i<\eta,v>}\,\mathfrak{K}_{\delta,\theta}'(z,v)=\\
&=(2\pi)^{d/2}\int_{\R^d}dv\,e^{-i<\eta,v>}\,\int_0^1ds\,\big[F_\theta(z)\cdot\big(\partial_z(\bb1\otimes\mathcal{F}^-)a\big)\big(z+s\delta\,F_\theta(z),v\big)\big]\\
&=\int_0^1ds\,\big[F_\theta(z)\cdot(\partial_za)\big(z+s\delta\,F_\theta(z),\eta\big)\big].
\end{align*}
As in Remark \ref{R-est-aFdelta}, we notice that
\beq
\nu_{n,m}(a'_{\delta,\theta})\,\leq\,\big(\underset{z\in\R^d}{\sup}\big|F_\theta(z)\big|\big)\,\underset{0\leq s\leq1}{\sup}\nu_{n+1,m}(a[F]_s)\leq\,M_F\,\theta^{-1}\,\underset{0\leq s\leq1}{\sup}\nu_{n+1,m}(a[F]_s)
\eeq
\end{proof}
\begin{lemma}
	The symbol $a''_{\delta,\theta}(z,\eta)$ associated to the kernel 
	$$
	\mathfrak{K}_{\delta,\theta}''(z,v):=\int_{\R^d}du\,\mW_\kappa(z-u)\,\Big(\int_0^1ds\,\big[F_\theta(u)\cdot\big(\partial_z\mathfrak{K}_0\big)\big(z+s\delta\,F_\theta(u),v\big)\big]\Big)
	$$
	belongs to $S^0_{0,0}(\R^d\times\R^d)$
	with seminorms bounded by $C\theta^{-1}$ uniformly for $(\delta,\theta)\in(0,1]^2$.
\end{lemma}
\begin{proof}
	We can write that: \begin{align*}
		a''_{\delta,\theta}(z,\eta)&=(2\pi)^d\int_{\R^d}dv\,e^{-i<\eta,v>}\,\mathfrak{K}_{\delta,\theta}''(z,v)=\\
		&=\int_0^1ds\,\int_{\R^d}du\,\mW_\kappa(z-u)\,\Big(\big[F_\theta(u)\cdot(\partial_z\,a)\big(z+s\delta\,F_\theta(u),\eta\big)\big]\Big).
	\end{align*}
	As in Remark \ref{R-est-aFdelta}, we notice that
	\beq
	\nu_{n,m}(a'_{\delta,\theta})\,\leq\,\big(\underset{u\in\R^d}{\sup}\big|F_\theta(u)\big|\big)\,\underset{0\leq s\leq1}{\sup}\nu_{n+1,m}(a[F]_s)\leq\,M_F\,\theta^{-1}\,\underset{0\leq s\leq1}{\sup}\nu_{n+1,m}(a[F]_s)
	\eeq
\end{proof}
Putting the above results together we conclude that:
\beq
\Big|\big\langle\mathfrak{K}^\circ_\delta\,,\,(\overline{\phi}\otimes\phi)\circ\Upsilon^{-1}\big\rangle_{\mathscr{S}(\R^d\times\R^d)}\,-\,\big\langle\mW[\mathfrak{K}^\circ_\delta]\,,\,(\overline{\phi}\otimes\phi)\circ\Upsilon^{-1}\big\rangle_{\mathscr{S}(\R^d\times\R^d)}\Big|\,\leq\,C(F)\delta\theta^{-1}.
\eeq

Let us consider now the 'outer region' integrals:
\begin{align}
&\big\langle\mathfrak{K}^\bot_\delta\,,\,(\overline{\phi}\otimes\phi)\circ\Upsilon^{-1}\big\rangle_{\mathscr{S}(\R^d\times\R^d)}\,-\,\big\langle\mW[\mathfrak{K}^\bot_\delta]\,,\,(\overline{\phi}\otimes\phi)\circ\Upsilon^{-1}\big\rangle_{\mathscr{S}(\R^d\times\R^d)}=\\ \nonumber
&\hspace*{9cm}=\ \delta\mathcal{I}_1[\phi](\delta,\theta,\kappa)\,+\,\delta^2\mathcal{I}_2[\phi](\delta,\theta,\kappa)
\end{align}
where:
\begin{align}\nonumber
&\mathcal{I}_1[\phi](\delta,\theta,\kappa):=\int_{\R^d}du\int_{\R^d}dz\,\mW_\kappa(z-u)\int_{\R^d}dv\,\overline{\phi(z+v/2)}\phi(z-v/2)\,\big(\nabla_z\mathfrak{K}_0\big)\big(z+\delta F_\theta^\bot(z),v\big)\,\times\\ \nonumber
&\hspace*{0.2cm}\times\,\Big[\big((z-u)\cdot\nabla F_\theta^\bot\big)(z)+(1/2)\int_0^1ds\,(1-s)\big((z-u)\otimes(z-u)(\nabla\otimes\nabla)\,F_\theta^\bot\big)\big(z+s(u-z)\big)\Big]
\end{align}
\begin{align}\nonumber
\mathcal{I}_2[\phi](\delta,\theta,\kappa)&:=\,(1/2)\int_{\R^d}du\int_{\R^d}dz\,\mW_\kappa(z-u)\int_{\R^d}dv\,\overline{\phi(z+v/2)}\phi(z-v/2)\,\times\\ \nonumber
&\hspace*{1cm}\times\,\int_0^1(1-s)ds\big((\nabla_z\otimes\nabla_z)\mathfrak{K}_0\big)\big(z+\delta F_\theta^\bot(z)+s\delta(F_\theta^\bot(u)-F_\theta^\bot(z)),v\big)\,\times\\ \nonumber
&\hspace*{1cm}\times\,\int_0^1dt\,\big((z-u)\cdot\nabla F_\theta^\bot\big)(z+t(u-z))\,\int_0^1dt'\,\big((z-u)\cdot\nabla F_\theta^\bot\big)(z+t'(u-z))
\end{align}

We may conclude that:
\beq
\big|\mathcal{I}_1[\phi](\delta,\theta,\kappa)\big|\,\leq\,\kappa^{-2}\theta^{1+\mu}\|\phi\|_{L^2}^2,\qquad\big|\mathcal{I}_2[\phi](\delta,\theta,\kappa)\big|\,\leq\,\kappa^{-2}\|\phi\|_{L^2}^2.
\eeq
\end{proof}

The result of this second Proposition clearly implies that:
\beq\label{E-2}\ 
\big|\widetilde{\mathcal{E}_+}(\delta)\,-\,\mathcal{E}_+(\delta)\big|\,=\,\mathcal{O}\big(\delta/\theta+\delta\kappa^{-2}\theta^{1+\mu}+\delta^2\kappa^{-2}\big).
\eeq

{\  Taking into account \eqref{E-1} and \eqref{E-2}}, in order to finish the proof of Theorem \ref{T-B}, we only have to make the following choices for our scaling parameters:
\begin{itemize}
	\item $\theta=\delta^{1-\rho}$ for some $\rho\in(0,1)$, so that $\delta\theta^{-1}=\delta^\rho$;
	\item $\kappa^{2}=\delta^\rho$ so that $\delta\kappa^{-2}\theta^{1+\mu}=\delta^{(1-\rho)+(1-\rho)(1+\mu)}=\delta^{(2+\mu)(1-\rho)}$ and $\delta^2\kappa^{-2}=\delta^{(2-\rho)}$.
	\item imposing $\rho=(2+\mu)(1-\rho)\in(0,1)$ means taking $\rho=(1+\mu)/(2+\mu)$.
\end{itemize}
\bigskip

\noindent\textbf{Acknowledgements:} The authors' work on this material was supported by the Independent Research Fund Denmark–Natural Sciences, grant DFF–10.46540/2032-00005B.

\bigskip \bigskip
{\footnotesize
\begin{tabular}{ll}
(H. D.~Cornean)   
		&  \textsc{Department of Mathematical Sciences, Aalborg University} \\ 
        	&   Thomas Manns Vej 23, 9220 Aalborg, Denmark\\
        	&  {E-mail address}: \href{mailto:cornean@math.aau.dk}{\texttt{cornean@math.aau.dk}}\\[10pt]
       
         (R.~Purice)   
       &  \textsc{"Simion Stoilow" Institute of Mathematics of the Romanian Academy} \\ 
        	&   {Calea Grivi\c{t}ei 21, Bucure{\c s}ti, P.O. Box 1-764, 014700, Romania}\\
        	&  {E-mail address}: \href{mailto:Radu.Purice@imar.ro}{\texttt{Radu.Purice@imar.ro}}\\
\end{tabular}
}

\end{document}